\documentclass[11pt,letterpaper]{amsart}
\usepackage{amsfonts,amsmath,amsthm,amssymb}

\usepackage{graphicx}
\usepackage{setspace}
\usepackage{thmtools}

\usepackage{dsfont}
\usepackage{grffile}

\usepackage{xcolor}
\usepackage{hyperref}
\hypersetup{
    colorlinks = true,
    linkcolor = {blue}, urlcolor={blue}
}

\definecolor{c1}{rgb}{0,0,1} 
\hypersetup{
    linkcolor= {c1}, 
    citecolor={c1}, 
    urlcolor={c1} 
}

\RequirePackage[numbers]{natbib}

\usepackage [english]{babel}
\usepackage [autostyle, english = american]{csquotes}
\MakeOuterQuote{"}

\newtheorem{thm}{Theorem}
\newtheorem{cor}{Corollary}
\newtheorem{lemma}{Lemma}

\newtheorem{prop}{Proposition}

\numberwithin{equation}{section}
\numberwithin{thm}{section}
\numberwithin{lemma}{section}
\numberwithin{cor}{section}
\numberwithin{prop}{section}

\begin{document}

{
  \title{\bf Zeta zero dependence and the critical line}
  \author{Gordon Chavez}

\date{}
  \maketitle
}

\begin{abstract}
On the critical line the conditional distribution of the zeta function's magnitude around zeta zeros exists and predicts the well-known pair correlation between nontrivial zeta zeros. However, this conditional distribution does not exist at most distances above or below any nontrivial zeta zeros that are off the critical line. This shows that the zeta function's magnitude cannot have vertical statistical structure at most distances around nontrivial zeta zeros off the critical line. 
The proofs of these results are straightforward, using only statistical properties of certain prime sums, elementary properties of normal and elliptical random variables, and the pole structure of the zeta function. These results readily generalize to L-functions.
\end{abstract}

MSC 2020 subject classification: Primary 11M06, 11M26; secondary 60F05.
Keywords: Riemann zeta function, L-function, nontrivial zeros, conditional distribution.

\tableofcontents

\section{Introduction}

The Riemann zeta function $\zeta(s)$ may be represented by the following sum over integers $n$ or product over prime numbers $p$ for $\textnormal{Re}(s)>1$:
\begin{equation}
\zeta(s)=\sum_{n=1}^{\infty}\frac{1}{n^{s}}=\prod_{p}\left(1-\frac{1}{p^{s}}\right)^{-1} \label{zeta}
\end{equation}
The former representation is called the Dirichlet series, while the latter representation is called the Euler product. $\zeta(s)$ may be extended by analytic continuation to all $s$ in $\mathbb{C}$ except $s=1$, as $\zeta(1)$ is a simple pole. Values of $s$ such that $\textnormal{Im}(s)>0$ and $\zeta(s)=0$, the nontrivial zeta zeros, are of great importance for number theory, as they are directly relevant for understanding the distribution of the prime numbers. The proof or disproof of the famous Riemann hypothesis (RH), which asserts that all of the zeta function's nontrivial zeros reside on the critical line, $\textnormal{Re}(s)=1/2$, is likely the most important open problem in number theory.

The importance of the critical line has led to a substantial amount of research on the behavior of the zeta function on this line. This research includes studies of the distribution of the zeta function's nonzero values on the critical line and the distribution of its zeros on the critical line. A foundational result in the former area is Selberg's central limit theorem \cite{selberg} \cite{radsound}, which states that, for $t$ uniformly distributed in $[T,2T]$ with $T\rightarrow \infty$,
\begin{equation}
\log\left|\zeta\left(\frac{1}{2}+it\right)\right| \overset{d} \to \mathcal{N}\left(0, \frac{1}{2}\log\log T\right). \label{selberg clt}
\end{equation}
That is, $\log\left|\zeta(1/2+it)\right|$ converges to a normally distributed random variable with mean zero and variance $\frac{1}{2}\log\log T$ under $t$ uniformly distributed in $[T,2T]$. 

A similarly foundational discovery regarding the distribution of the zeta zeros on the critical line was made by Montgomery \cite{montgomery}, who gave the first results describing pair-correlation between zeros, assuming RH. His results predicted a short-range repulsion between consecutive zeros on the critical line. Odlyzko \cite{odlyzko} later calculated large numbers of zeta zeros and empirically corroborated Montgomery's predictions. Dyson \cite{dyson} famously pointed out that Montgomery's result was equivalent to that for the pair-correlation of eigenvalues of random unitary matrices. This connection with random matrix theory has been extremely useful for understanding statistical dependence in zeros of the zeta function \cite{snaith} and its generalizations  \cite{hejhal} \cite{rudnicksarnak94} \cite{rudnicksarnak96} \cite{katzsarnak}. This connection has also found many other applications, e.g., by Keating \& Snaith \cite{keating snaith}, who gave influential conjectures on the moments of the zeta function on the critical line, and by Fyodorov, Hiary, \& Keating (FHK) \cite{fhk} \cite{fk}, who gave a notable conjecture on the distribution of the zeta function's extreme values on the critical line.

Bourgade \cite{bourgade2010} provided an important extension to Selberg's and Montgomery's results by showing that, for $t$ uniformly distributed in $[T,2T]$, the covariance of $\log\left|\zeta\left(1/2+it\right)\right|$ and $\log\left|\zeta\left(1/2+i\left(t+\Delta\right)\right)\right|$ is 
$$\approx -\frac{1}{2}\log \left|\Delta\right|$$
for mesoscopic distances $\Delta$ where $1/\log T\ll \left|\Delta\right| \ll 1$, showing that \newline $\log\left|\zeta(1/2+it)\right|$ has logarithmic correlations over such distances. This result then confirmed suggestions by Coram \& Diaconis \cite{coram diaconis} that, along with the microscopic repulsion  ($\left|\Delta\right|<1/\log T$) between zeta zeros predicted by Montgomery, there is also a mesoscopic repulsion between zeros. The logarithmic correlation structure of $\log\left|\zeta(1/2+it)\right|$ was an important motivation for FHK's conjecture, and it served as an important link to the theory of branching random walks. The latter was applied by Arguin, Belius, \& Harper \cite{branch} to show that the leading terms of FHK's conjecture hold for a random model of the zeta function, applied by Arguin, Belius, Bourgade, Radziwill, \& Soundararajan \cite{maxshort} to verify the leading order of FHK's conjecture for the zeta function itself, by Arguin, Bourgade, \& Radziwill \cite{fhk 1} to prove the upper bound in FHK's conjecture, and by Arguin, Oimet, \& Radziwill \cite{shortgen} to provide new results describing the moments and maxima of the zeta function on the critical line, proving a related conjecture of Fyodorov \& Keating \cite{fk}, and generalizing the results in \cite{maxshort}. 

Random matrix theory is closely connected to the field of quantum chaos \cite{berry 77} \cite{bohigas et al}, and it was ideas from this field that first illuminated long-range or macroscopic ($\left|\Delta\right| \geq 1$) statistical dependence in the zeta zeros. In particular, Bogomolny \& Keating \cite{bk951} \cite{bk952} \cite{bk96} gave the first predictions for such dependence in the Riemann zeros by utilizing semi-classical techniques from the field of quantum chaos along with the Hardy-Littlewood conjecture from number theory. Their results describe an effect where differences between zeros tend to avoid the imaginary parts of the low-lying zeta zeros themselves. Conrey \& Snaith \cite{conreysnaith} showed similar results from a conjecture on the ratios of L-functions given by Conrey, Farmer, \& Zirnbauer \cite{cfz} and Rodgers \cite{rodgers} proved the effect under RH. Ford \& Zaharescu \cite{fordzah} then used uniform versions of Landau's formula \cite{landau} \cite{goneklandau} \cite{fordzah2} to give unconditional proof of the effect in zeta and L-function zeros on the critical line, explaining numerical observations of the effect in L-function zeros by Perez-Marco \cite{pm11}. 

The formal results above concerning the zeta zeros either only consider zeros on the critical line, assume RH, or include contributions from zeros off the critical line in an error term. A question then arises: would nontrivial zeta or L-function zeros off the critical line have the same vertical statistical dependence structure as that described by the results above? There is little research on this topic. Gonek \cite{gonek} applied a finite Euler product approximation of the zeta function \cite{gonekhugheskeating} \cite{gonek12} to predict that any such zeros off the critical line have a structurally different generating process from those on the critical line, which suggests the answer is no. In this paper we apply probabilistic methods similar to those used in much of the research described above as well as other, related work \cite{matsumoto} \cite {harper} \cite{largedev}, to show that such zeros likely would not have any vertical statistical structure.

On the critical line we can prove the following result describing the conditional distribution of $\log\left|\zeta\left(1/2+i\left(t+\Delta\right)\right)\right|$ when $\zeta\left(1/2+it\right)= 0$\footnote{For a random variable $X$ and event $Y$, we denote by $X|Y$ the random variable $X$ given that the event $Y$ has occurred.}:
\begin{thm}
If $t$ is uniformly distributed in $[T,2T]$ with $\Delta \in [T-t, 2T-t]$, then as $T\rightarrow \infty$,
\begin{align}
\left. \log \left|\zeta\left(\frac{1}{2}+i\left(t+\Delta\right)\right)\right| \middle \vert \zeta\left(\frac{1}{2}+it\right)=0 \right.\nonumber \\ \overset{d} \to \mathcal{N}\left(-\sum_{k=1}^{\infty}\frac{\textnormal{Re}\mathcal{P}\left(k\left(1+i\Delta\right)\right)}{k^{2}}, \frac{1}{2}\log\log T\right) \label{critical line}
\end{align}
where $\mathcal{P}(s)$ denotes the prime zeta function. \label{thm intro}
\end{thm}
This limit theorem predicts the statistical dependence between zeta zeros described above. In particular, it predicts an "average gap" around zeta zeros, which agrees with the zeros' known short-range dependence, and, since the conditional mean in (\ref{critical line}) has local maxima at values of $\Delta$ approximately equal to the Riemann zeros' imaginary parts, it predicts the long-range repulsion between zeta zero differences and the imaginary parts of the low-lying zeta zeros (see Figure \ref{fig0}). We will give further details on this result below. 

\begin{figure}
\centering
\includegraphics[width=\linewidth]{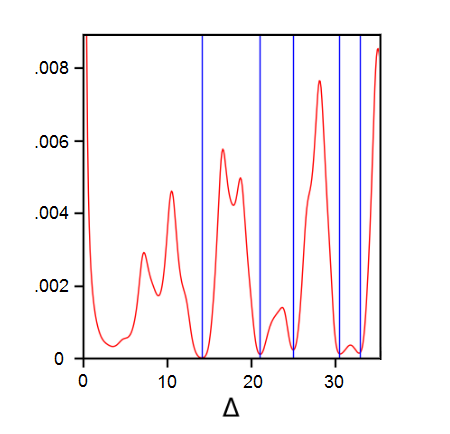}
\caption{A plot (Red) of the conditional probability $\mathbb{P}\left\{\log\left|\zeta\left(1/2+i\left(t+\Delta\right)\right)\right|\leq -3\sigma \middle \vert \zeta\left(1/2+it\right)=0\right\}
$ approximated using the distribution described by (\ref{critical line}) with $\sigma=\frac{1}{2}\log\log t$, where $t$ is the imaginary part of the 100,000th Riemann zero. Vertical lines (Blue) are positioned at values of $\Delta$ such that $\zeta(1/2+i\Delta)=0$. Since large negative values of $\log\left|\zeta(s)\right|$ are a necessary condition for $s$ to be the position of a zeta zero, this result demonstrates the statistical dependence between zeta zeros on the critical line described above.}
\label{fig0}
\end{figure}

Perhaps most notably, we show that statistical dependence like that described by (\ref{critical line}) for zeta zeros on the critical line cannot exist for zeros that are off the critical line. In particular, for $\sigma>1/2$, we show that the conditional distribution of $\log\left|\zeta\left(\sigma+i\left(t+\Delta\right)\right)\right|$ does not exist when $\zeta\left(\sigma+it\right)= 0$ for most values of $\Delta$, i.e., the conditional moments are infinite and the conditional probability density function is identically zero. This shows that the zeta function's magnitude cannot have vertical statistical structure at most distances around nontrivial zeta zeros that are off the critical line. 

These results can be proven relatively easily, using only statistical properties of certain prime sums, elementary properties of normal and elliptical random variables, and the pole structure of the zeta function. For several related results we additionally apply elementary properties of generalized hyperbolic random variables. All of these results readily generalize to L-functions.

\section{Preliminaries}
\label{prelims}

We first describe the prime zeta function, $\mathcal{P}\left(s\right)$, which has the expressions
\begin{equation}
\mathcal{P}\left(s\right)=\sum_{p}\frac{1}{p^{s}} \label{prime zeta sum}
\end{equation}
and 
\begin{equation}
\mathcal{P}\left(s\right)=\sum_{n=1}^{\infty}\frac{\mu(n)}{n}\log\zeta\left(ns\right),\label{prime zeta ac}
\end{equation}
where 
\begin{equation}
\mu(n)=\begin{cases} 0; & \textnormal{$n$ has a repeated prime factor} \\ (-1)^{\omega(n)}; & \textnormal{$n$ has $\omega(n)$ distinct prime factors} \end{cases}  \label{mobius defn}
\end{equation}
is the well-known M{\"o}bius function. The sum over primes (\ref{prime zeta sum}) is absolutely convergent for $\textnormal{Re}(s)>1$ and the analytic continuation (\ref{prime zeta ac}), which is derived using the technique of M{\"o}bius inversion, is convergent everywhere in $\textnormal{Re}(s)>0$ except at the points $ns=1$, which arise from the simple pole $\zeta(1)$, and at the points $ns$ where $\zeta\left(ns\right)=0$ \cite{glaisher} \cite{froberg}. Note that, since $\zeta(s)$ has no zeros in $\textnormal{Re}(s)\geq 1$, the singularities of (\ref{prime zeta ac}) in $1/2 \leq \textnormal{Re}(s)<1$ with $\textnormal{Im}(s) \neq 0$ unambiguously define the positions of the nontrivial zeta zeros. 

We make some important notes about the statistical properties of (\ref{prime zeta sum}). Under $t$ uniformly distributed in $[a,b]$ with $b-a\rightarrow \infty$, the summands of (\ref{prime zeta sum})'s $\textnormal{Re}\mathcal{P}(\sigma+it)$ are statistically independent\footnote{See Appendix \ref{app independence} for sketch of proof.}. This is directly related to the fact that, due to the fundamental theorem of arithmetic, the characteristic function of (\ref{prime zeta sum})'s $\textnormal{Re}\mathcal{P}(\sigma+it)$ has the factorizable form
\begin{equation}
\varphi\left(\lambda\right)=\prod_{p}J_{0}\left(\frac{\lambda}{p^{\sigma}}\right)=\prod_{p}\varphi_{p}\left(\lambda\right), \label{chf0}
\end{equation}
where $\varphi_{p}\left(\lambda\right)$ here denotes the characteristic function of $\cos\left(t\log p\right)/p^{\sigma}$ and
\begin{equation}
J_{0}(z)=\sum_{m=0}^{\infty}\frac{\left(-1\right)^{m}}{\left(m!\right)^{2}}\left(\frac{z}{2}\right)^{2m} \label{bessel0}
\end{equation}
is the 0th order Bessel function of the first kind \cite{laurincikas}. We apply this statistical independence property first here by multiplying $\textnormal{Re}\mathcal{P}\left(k\left(\sigma+i(t+\Delta)\right)\right)$ and $\textnormal{Re}\mathcal{P}\left(j\left(\sigma+it\right)\right)$ and applying the expected value\footnote{For $t$ uniformly distributed in $[a,b]$, the expected value of a given function $f(.)$ is defined $E\left\{f\left(t\right)\right\}=\frac{1}{b-a}\int_{a}^{b}f\left(x\right)dx$.} to give the following useful result describing the covariance of $\textnormal{Re}\mathcal{P}\left(j\left(\sigma+it\right)\right)$ and $\textnormal{Re}\mathcal{P}\left(k\left(\sigma+i(t+\Delta)\right)\right)$:

\begin{lemma}
If $t$ is uniformly distributed in $[T,2T]$, then as $T\rightarrow \infty$, 
\begin{align}
E\left\{\textnormal{Re}\mathcal{P}\left(k\left(\sigma+i(t+\Delta)\right)\right)\textnormal{Re}\mathcal{P}\left(j\left(\sigma+it\right)\right)\right\} \rightarrow \begin{cases} \frac{1}{2}\textnormal{Re}\mathcal{P}\left(k(2\sigma+i\Delta)\right); &  k=j \\ 0; & k \neq j\end{cases} 
\label{covar general}
\end{align}
for $\Delta \in \mathbb{R}$ and $k,j \in \mathbb{N}$.
\label{cov gen lemma}
\end{lemma}
\begin{proof}
We begin by noting from the Fubini and Tonelli theorems that if
\begin{equation}
\sum_{n}\frac{1}{b-a}\int_{a}^{b}\left|f_{n}(t)\right|dt<\infty \label{ft cond}
\end{equation}
for general functions $f_{n}(t)$, then 
\begin{equation}
\frac{1}{b-a}\int_{a}^{b}\sum_{n}f_{n}(t)dt=\sum_{n}\frac{1}{b-a}\int_{a}^{b}f_{n}(t)dt. \label{ft}
\end{equation}
We then apply (\ref{prime zeta sum}) to note that, for any $t, \Delta \in \mathbb{R}$, $k\geq j \in \mathbb{N}$, and $j\sigma>1$,
\begin{align}
\textnormal{Re}\mathcal{P}\left(k\left(\sigma+i(t+\Delta)\right)\right)\textnormal{Re}\mathcal{P}\left(j\left(\sigma+it\right)\right)=\sum_{p}\frac{\cos\left(k(t+\Delta)\log p\right)\cos\left(jt\log p\right)}{p^{(k+j)\sigma}} \nonumber \\ +2\sum_{p}\frac{\cos\left(k(t+\Delta)\log p\right)}{p^{k\sigma}}\sum_{q<p}\frac{\cos\left(jt \log q\right)}{q^{j\sigma}}. \label{gen pz multiply}
\end{align}
We then note that the series on the right-hand side of (\ref{gen pz multiply}) satisfy the absolute convergence property
\begin{align}
\sum_{p}\frac{\left| \cos\left(k(t+\Delta)\log p\right)\cos\left(jt\log p\right)  \right|}{p^{(k+j)\sigma}} \nonumber \\ +2\sum_{p}\frac{\left| \cos\left(k(t+\Delta)\log p\right) \right|}{p^{k\sigma}}\sum_{q<p}\frac{\left| \cos\left(jt \log q\right) \right|}{q^{j\sigma}} \nonumber \\ 
\leq \sum_{p}\frac{1}{p^{(k+j)\sigma}}+2\sum_{p}\frac{1}{p^{k\sigma}}\sum_{q<p}\frac{1}{q^{j\sigma}} < \mathcal{P}((k+j)\sigma)+2\mathcal{P}(k\sigma)\mathcal{P}(j\sigma). \label{gen bound}
\end{align}
Note that (\ref{gen bound}) shows that (\ref{gen pz multiply}) satisfies (\ref{ft cond}). Therefore the expected value can be applied term-by-term to (\ref{gen pz multiply}).

We then apply the expected value with $a=T$, $b=2T$, and $T \rightarrow \infty$ and note that, by the statistical independence shown in Appendix \ref{app independence}, the expected value of the double-summation in (\ref{gen pz multiply}) always vanishes because $E\left\{\cos\left(\omega t+\theta\right)\right\}=0$ for all $\omega, \theta \in \mathbb{R}$. This latter vanishing property combined with the fact that
\begin{align}
\cos\left(k(t+\Delta)\log p\right)\cos\left(jt\log p\right) \nonumber \\ =\frac{1}{2}\left(\cos\left((k-j)t\log p+\Delta \log p\right)+\cos\left(\left((k+j)t \log p+\Delta \log p\right)\right)\right) \label{almost done ugh}
\end{align}
shows that the expected value of (\ref{gen pz multiply})'s first summation vanishes unless $k=j$. This gives
\begin{align}
E\left\{\textnormal{Re}\mathcal{P}\left(k\left(\sigma+i(t+\Delta)\right)\right)\textnormal{Re}\mathcal{P}\left(j\left(\sigma+it\right)\right)\right\} \rightarrow\begin{cases} \frac{1}{2}\sum_{p}\frac{\cos\left(k\Delta \log p\right)}{p^{2k\sigma}}; &  k=j \\ 0; & k \neq j\end{cases}
\label{pre covar general}
\end{align}
which completes the proof.
\end{proof}
Note that the prime sum in (\ref{pre covar general}) is convergent for $2k\sigma>1$, which is a much weaker constraint than $k\sigma>1$. Therefore applying the expected value as above serves as an \textit{analytic continuation} to a larger domain for $\sigma$. Furthermore, (\ref{covar general}) may be extended to the still larger domain in $\sigma>0$ allowed by the analytic continuation (\ref{prime zeta ac}). We lastly note that (\ref{covar general}) with $k=j=1$ gives the following result for $\textnormal{Re}\mathcal{P}\left(\sigma+it\right)$'s autocovariance function:
\begin{equation}
R_{\mathcal{P}}(\sigma,\Delta)=E\left\{\textnormal{Re}\mathcal{P}\left(\sigma+i(t+\Delta)\right)\textnormal{Re}\mathcal{P}\left(\sigma+it\right)\right\} \rightarrow \frac{1}{2}\textnormal{Re}\mathcal{P}\left(2\sigma+i\Delta\right).
 \label{cov gen}
\end{equation}

\subsection{Elliptical distributions}
We next note that (\ref{chf0}) is an even function since $J_{0}(x)$ is an even function. This implies that $\textnormal{Re}\mathcal{P}(\sigma+it)$'s probability density function $p(x)$ is even as well. Relatedly, it is clear from (\ref{chf0})-(\ref{bessel0}) that there exists a function $\phi(.)$ such that 
$$
\varphi(\lambda)=\phi\left(\lambda^{2}\right),
$$
which, with (\ref{cov gen}), shows that $\textnormal{Re}\mathcal{P}(\sigma+it)$ and $\textnormal{Re}\mathcal{P}(\sigma+i(t+\Delta))$ have an \textit{elliptical} joint distribution, which is a large class of distributions that includes the multivariate normal distribution. From the general results of Fang, Kotz, \& Ng \cite{kotz} (Theorem 2.18 pg. 45, Eqn. 2.43 pg. 46) the conditional probability density function of an elliptical random variable $X(t+\Delta)$ given the value of $X(t)$ has the form
\begin{equation}
p_{t+\Delta|t}(x)=\frac{1}{\sqrt{h(\Delta)}}g\left(\frac{\left(x-\mu_{t+\Delta|t}\right)^{2}}{h(\Delta)}\right) \label{cond dens gen}
\end{equation}
where $g(.)$ is a non-negative function, 
\begin{equation}
h(\Delta)=R(0)\left(1-\left(\frac{R(\Delta)}{R(0)}\right)^{2}\right), \label{variance factor}
\end{equation}
$R(\Delta)=E\left\{X(t+\Delta)X(t)\right\}$, $R(0)$ is thus the unconditional variance, and the conditional expectation\footnote{The conditional moments here are defined $E\left\{X^{m}(t+\Delta)|X(t)\right\}=\int_{-\infty}^{\infty}x^{m}p_{t+\Delta|t}(x)dx$.} is given by
\begin{equation}
\mu_{t+\Delta|t}=E\left\{X(t+\Delta)|X(t)\right\}=\frac{R(\Delta)}{R(0)}X(t). \label{gen cond mean}
\end{equation}
Additionally, from \cite{kotz} (Theorem 2.18 pg. 45), one can show that the conditional variance has the form 
\begin{equation}
\textnormal{var}\left\{X(t+\Delta)|X(t)\right\}=h\left(\Delta\right)V\left(X(t)\right), \label{gen cond var}
\end{equation}
where $V(.)$ is a function that depends on the specific distribution of $X(t)$ and $X(t+\Delta)$. In the case of $X(t)$ and $X(t+\Delta)$ normally distributed, $V\left(.\right)=1$. In general, however, this is not the case and (\ref{gen cond var}) is dependent on the value of $X(t)$. 

\subsection{Generalized hyperbolic distributions}
We next describe the tails of distributions of series with the form (\ref{prime zeta sum}) with $s=\sigma+it$ that additionally satisfy the convergence property 
\begin{equation}
\sum_{p}\frac{1}{p^{2\sigma}}<\infty, \label{sq conv}
\end{equation}
i.e. $\sigma>1/2$. In particular, if (\ref{prime zeta sum}) satisfies (\ref{sq conv}), then (\ref{prime zeta sum})'s probability density function $p(x)$ satisfies
\begin{equation}
p(x)=O\left(e^{-r |x|}\right) \label{density asy}
\end{equation}
for some $r\geq 1$ as $|x| \rightarrow \infty$. This may be shown from Laurincikas \cite{laurincikas} (Eqn. 6.17 pg. 47). We then note that the \textit{generalized hyperbolic} distributions are a large family that includes many distributions that are symmetric, elliptical, and satisfy the tail constraint (\ref{density asy}). If $X(t)$ and $X(t+\Delta)$ have such a distribution, then one can apply the general results of Blaesild \& Jensen \cite{blaesild81} ((1) pg.47, Theorem 1 (b) pg. 49-50) to show that the conditional probability density of $X(t+\Delta)$ given $X(t)$ has the form
\begin{align}
p_{t+\Delta|t}(x)=\sqrt{\frac{\alpha'}{2\pi}}\left(\frac{1}{\delta_{t+\Delta|t}}\right)^{\lambda-1/2}\frac{K_{\lambda-1}\left(\alpha'\left(\delta_{t+\Delta|t}^{2}+\frac{\left(x-\mu_{t+\Delta|t}\right)^{2}}{R(0)h(\Delta)}\right)^{1/2}\right)}{K_{\lambda-1/2}(\alpha' \delta_{t+\Delta|t})\left(\delta_{t+\Delta|t}^{2}+\frac{\left(x-\mu_{t+\Delta|t}\right)^{2}}{R(0)h(\Delta)}\right)^{\frac{1-\lambda}{2}}}, \label{gen hyp cond dens}
\end{align}
where $K_{\nu}(z)$ is the $\nu$th order modified Bessel function of the second kind, 
\begin{equation}
\alpha'=\alpha\sqrt{R(0)}, \hspace{.4cm}
\delta_{t+\Delta|t}=\sqrt{\frac{\delta^{2}}{R(0)}+\frac{X^{2}(t)}{R^{2}(0)}}, \label{gen hyp params}
\end{equation}
and $\lambda$, $\alpha$, and $\delta$ are given parameters. Additionally from \cite{blaesild81} (pg. 50-51) one can show that the conditional variance has the form (\ref{gen cond var}) with 
\begin{equation}
V\left(X(t)\right)=\frac{1}{\alpha}\sqrt{\delta^{2}+\frac{X^{2}(t)}{R(0)}}\frac{K_{\lambda+1/2}\left(\alpha\sqrt{\delta^{2}+\frac{X^{2}(t)}{R(0)}}\right)}{K_{\lambda-1/2}\left(\alpha\sqrt{\delta^{2}+\frac{X^{2}(t)}{R(0)}}\right)}.\label{gen hyp var}
\end{equation}
We then note the asymptotic result \cite{bessel} (DLMF 10.17.1, 10.40.2)
\begin{equation}
K_{\nu}(z)\sim \sqrt{\frac{\pi}{2z}}e^{-z}, \label{m bessel identity}
\end{equation}
which, with (\ref{variance factor}), (\ref{gen cond var}), and (\ref{gen hyp var}) shows that, for the generalized hyperbolic family,
\begin{equation}
\textnormal{var}\left\{X(t+\Delta)|X(t)\right\} \sim \left|X(t)\right|\frac{\sqrt{R(0)}}{\alpha}\left(1-\left(\frac{R(\Delta)}{R(0)}\right)^{2}\right) \label{gen hyp var asy}
\end{equation}
as $\left|X(t)\right|\rightarrow \infty$. 

We will use the properties of $\mathcal{P}(s)$ described above as well as the properties of elliptical and generalized hyperbolic random variables to study the conditional distribution of $\log\left|\zeta(\sigma+i(t+\Delta))\right|$ when $\zeta(\sigma+it)=0$. We will see that the conditional distribution exists for $\sigma=1/2$ but not for $\sigma>1/2$.

\section{Application to zeta and L-functions}
We begin by using Lemma \ref{cov gen lemma} to give the following result describing $\log\left|\zeta\left(\sigma+it\right)\right|$'s autocovariance function:

\begin{lemma}
If $t$ is uniformly distributed in $[T,2T]$, then as $T\rightarrow \infty$,  
\begin{align}
R_{\log \zeta}(\sigma,\Delta)=E\left\{\log\left|\zeta\left(\sigma+i(t+\Delta)\right)\right|\log\left|\zeta\left(\sigma+it\right)\right|\right\} \nonumber \\ \rightarrow\frac{1}{2}\sum_{k=1}^{\infty}\frac{\textnormal{Re}\mathcal{P}\left(k(2\sigma+i\Delta)\right)}{k^{2}} \label{logzeta cov}
\end{align}
for $\Delta \in \mathbb{R}$.  \label{logzeta cov lemma}
\end{lemma}
\begin{proof}
We begin by considering $\sigma>1$ so that we may apply the Euler product and Taylor expansion for $\log\left|\zeta(s)\right|$ to write
\begin{align}
\log\left|\zeta\left(\sigma+i(t+\Delta)\right)\right|\log\left|\zeta\left(\sigma+it\right)\right| \nonumber \\ =\sum_{k=1}^{\infty}\frac{\textnormal{Re}\mathcal{P}\left(k\left(\sigma+i(t+\Delta)\right)\right)\textnormal{Re}\mathcal{P}\left(k\left(\sigma+it\right)\right)}{k^{2}} \nonumber \\
+2\sum_{k=1}^{\infty}\sum_{j<k}\frac{\textnormal{Re}\mathcal{P}\left(k\left(\sigma+i(t+\Delta)\right)\right)\textnormal{Re}\mathcal{P}\left(j\left(\sigma+it\right)\right)}{kj}  \label{app error 3}
\end{align}
We next note that 
\begin{align}
\sum_{k=1}^{\infty}\frac{\left|\textnormal{Re}\mathcal{P}\left(k\left(\sigma+i(t+\Delta)\right)\right)\textnormal{Re}\mathcal{P}\left(k\left(\sigma+it\right)\right)\right|}{k^{2}} \leq \sum_{k=1}^{\infty}\frac{\mathcal{P}^{2}(k\sigma)}{k^{2}}<\infty \label{term 1 upper bound}
\end{align}
and 
\begin{align}
2\sum_{k=1}^{\infty}\sum_{j<k}\frac{\left|\textnormal{Re}\mathcal{P}\left(k\left(\sigma+i(t+\Delta)\right)\right)\textnormal{Re}\mathcal{P}\left(j\left(\sigma+it\right)\right)\right|}{kj} \nonumber \\ \leq 2\sum_{k=1}^{\infty}\frac{\mathcal{P}(k\sigma)}{k}\sum_{j<k}\frac{\mathcal{P}(j \sigma)}{j}< 2\mathcal{P}(\sigma)\sum_{k=1}^{\infty}\mathcal{P}(k\sigma)<\infty. \label{term 2 upper bound}
\end{align}
These upper bounds may be easily shown to be finite from the asymptotic decay rates implied by (\ref{prime zeta sum}) and the ratio test. The finite bounds (\ref{term 1 upper bound}) and (\ref{term 2 upper bound}) show that (\ref{app error 3})'s series satisfy the condition (\ref{ft cond}). We may therefore apply the expectation to (\ref{app error 3}) term-by-term with $a=T$, $b=2T$, and $T\rightarrow \infty$. We do so and apply Lemma \ref{cov gen lemma} to show that the expected value of (\ref{app error 3})'s second term vanishes while the expected value of its first term is given by (\ref{logzeta cov}).
\end{proof}
We note that, similarly to Lemma \ref{cov gen lemma}, Lemma \ref{logzeta cov lemma}'s result (\ref{logzeta cov}) is convergent using the prime sum (\ref{prime zeta sum}) for all $\sigma>1/2$, and it extends to a still larger domain in $\sigma>0$ using the analytic continuation (\ref{prime zeta ac}). We will next apply the above results to study the statistical behavior of $\log\left|\zeta(s)\right|$ both on and off the critical line.

\subsection{On the critical line}
We first show that, for $\sigma=1/2$, Lemma \ref{logzeta cov lemma}'s result (\ref{logzeta cov}) reproduces the logarithmic correlations described by Bourgade \cite{bourgade2010} as well as Fyodorov, Hiary, \& Keating \cite{fk} \cite{fhk}. Showing this will use several principles that we will also apply in the proof of Theorem \ref{thm intro} below. We first note that, since $\mathcal{P}(s)=O(1)$ for all $\textnormal{Re}(s)>1$, (\ref{logzeta cov}) with $\sigma=1/2$ satisfies
$$
R_{\log \zeta}\left(1/2,\Delta\right)=\frac{1}{2}\textnormal{Re}\mathcal{P}\left(1+i\Delta\right)+O(1),
$$
which we can expand using (\ref{prime zeta ac}) to write
\begin{equation}
R_{\log \zeta}\left(1/2,\Delta\right)=\frac{1}{2}\log\left|\zeta\left(1+i\Delta\right)\right|+\frac{1}{2}\sum_{n=2}^{\infty}\frac{\mu(n)}{n}\log\left|\zeta\left(n\left(\sigma+i\Delta\right)\right)\right|+O(1). \label{420.1}
\end{equation}
We next make the general note that, for $\sigma \geq 1/2$ and $\tau \in \mathbb{R}$, 
\begin{equation}
\sum_{n=2}^{\infty}\frac{\mu(n)}{n}\log\zeta\left(n\left(\sigma+i\tau\right)\right)=O(1) \label{over with}
\end{equation}
because $\zeta(s)$ has no poles with $\textnormal{Re}(s)>1$ and the summands in (\ref{over with}) are $O\left(\frac{1}{n2^{n\sigma}}\right)$ as $n\rightarrow \infty$. 
Hence 
\begin{equation}
R_{\log \zeta}\left(1/2,\Delta\right)=\frac{1}{2}\log\left|\zeta\left(1+i\Delta\right)\right|+O(1). \label{420.2}
\end{equation}
We then note that $\zeta(s)$ has the following Laurent series expansion around its simple pole at $s=1$:
\begin{equation}
\zeta(s)=\frac{1}{s-1}+\sum_{k=0}^{\infty}\frac{(-1)^{k}}{k!}\gamma_{k}(s-1)^{k}, \label{laurent basic}
\end{equation}
where the $\gamma_{k}$ are the Stieltjes constants. It is clear from (\ref{laurent basic}) that $\zeta\left(1+i\Delta\right)\approx \frac{1}{i\Delta}$ for small $\left|\Delta\right|$ and hence, by (\ref{420.2}),
\begin{equation}
R_{\log \zeta}\left(1/2,\Delta\right)\approx -\frac{1}{2}\log \left|\Delta\right|, \label{420.3}
\end{equation}
reproducing logarithmic correlations over short distances.

We next will directly use Selberg's limit theorem (\ref{selberg clt}) as well as Lemma \ref{logzeta cov lemma} to prove Theorem \ref{thm intro}.

\label{thm intro proof}
\begin{proof}
Let $t$ be uniformly distributed in $[T,2T]$ with $\Delta \in [T-t,2T-t]$. By (\ref{selberg clt}), $\log\left|\zeta(1/2+it)\right|$ and $\log\left|\zeta\left(1/2+i\left(t+\Delta\right)\right)\right|$ are normally distributed for large $T$. Thus the conditional mean has the limiting form (\ref{gen cond mean}), which we combine with (\ref{logzeta cov}) to write 
\begin{align}
E\left\{\log\left|\zeta\left(1/2+i\left(t+\Delta\right)\right)\right| | \log\left|\zeta\left(1/2+it\right)\right|   \right\} \nonumber \\ \sim \frac{R_{\log\zeta}(1/2,\Delta)}{R_{\log\zeta}(1/2,0)}\log\left|\zeta(1/2+it)\right| 
=\frac{\sum_{k=1}^{\infty}\frac{\textnormal{Re}\mathcal{P}\left(k(1+i\Delta)\right)}{k^{2}}}{\sum_{k=1}^{\infty}\frac{\mathcal{P}(k)}{k^{2}}}\log\left|\zeta(1/2+it)\right|.
\label{app2 cond mean}
\end{align}
We then apply (\ref{prime zeta ac}) and use (\ref{over with}) and (\ref{laurent basic}) to note that
\begin{align}
\sum_{k=1}^{\infty}\frac{\mathcal{P}(k)}{k^{2}}=\mathcal{P}(1)+O(1)=\log\zeta(1)+\sum_{n=2}^{\infty}\frac{\mu(n)}{n}\log\zeta(n)+O(1) \nonumber \\ =\log \zeta(1)+O(1)=\log\left( \lim_{z\rightarrow 1^{+}}\frac{1}{z-1}\right)+O(1)=\lim_{z\rightarrow 0^{+}}\log \left(\frac{1}{z}\right)+O(1). \label{laurent}
\end{align}
We additionally note that when $\zeta(1/2+it)=0$
\begin{equation}
\log\left|\zeta\left(1/2+it\right)\right|=\lim_{z\rightarrow 0^{+}}\log \left(z\right). \label{there it is}
\end{equation}
Substituting the two limits (\ref{laurent})-(\ref{there it is}) into (\ref{app2 cond mean}) and cancelling gives the conditional expectation 
\begin{equation}
E\left\{\log\left|\zeta\left(1/2+i\left(t+\Delta\right)\right)\right| | \zeta\left(1/2+it\right)=0   \right\}\sim -\sum_{k=1}^{\infty}\frac{\textnormal{Re}\mathcal{P}\left(k(1+i\Delta)\right)}{k^{2}}. \label{app2 cond mean final}
\end{equation}

We next consider the conditional variance which, by (\ref{selberg clt}), has the limiting form (\ref{gen cond var}) where $V(.)=1$ and $h(\Delta)$ has the form (\ref{variance factor}). We also note from Selberg's central limit theorem that the leading factor $R(0)$ in (\ref{variance factor}) may be replaced with $\sim \frac{1}{2}\log\log T$. We apply these facts to write
\begin{align}
\textnormal{var}\left\{\log\left|\zeta\left(1/2+i\left(t+\Delta\right)\right)\right| | \log\left|\zeta\left(1/2+it\right)\right| \right\} \nonumber \\ \sim \left(\frac{1}{2}\log\log T\right)\left(1-\left(\frac{R_{\log \zeta}(1/2,\Delta)}{R_{\log \zeta}(1/2,0)}\right)^{2}\right) \nonumber \\ 
=\left(\frac{1}{2}\log\log T\right)\left(1-\left(\frac{\sum_{k=1}^{\infty}\frac{\textnormal{Re}\mathcal{P}\left(k(1+i\Delta)\right)}{k^{2}}}{\sum_{k=1}^{\infty}\frac{\mathcal{P}\left(k\right)}{k^{2}}}\right)^{2}\right)
\label{app2 var}
\end{align}
We then note from (\ref{prime zeta ac}) that $\mathcal{P}(1+i\Delta)$ and hence the numerator in (\ref{app2 var})'s second term,
$$
\left(\sum_{k=1}^{\infty}\frac{\textnormal{Re}\mathcal{P}\left(k(1+i\Delta)\right)}{k^{2}}\right)^{2},
$$
has no singularities, i.e., only takes finite values with $\Delta\neq 0$ while the denominator is $\left(\mathcal{P}(1)+O(1)\right)^{2}$, which, by (\ref{laurent}), diverges. We thus conclude that
\begin{equation}
\textnormal{var}\left\{\log\left|\zeta\left(1/2+i\left(t+\Delta\right)\right)\right| | \log\left|\zeta\left(1/2+it\right)\right| \right\} \sim \left(\frac{1}{2}\log\log T\right)\left(1+o(1)\right). \label{app2 var final}
\end{equation}
This completes the proof. 
\end{proof}

\subsection{Off the critical line}
By (\ref{cov gen}) the unconditional variance of $\textnormal{Re}\mathcal{P}\left(\sigma+it\right)$ for $\sigma>1/2$ is given by
\begin{equation}
\frac{\mathcal{P}\left(2\sigma\right)}{2}=\frac{1}{2}\sum_{p}\frac{1}{p^{2\sigma}}=O(1). \label{always conv}
\end{equation}
Since the variance (\ref{always conv}) is convergent, in contrast to when $\sigma=1/2$, a central limit theorem does not hold for $\textnormal{Re}\mathcal{P}\left(\sigma+it\right)$ with $\sigma>1/2$. Relatedly, we can show the following result describing the nonexistence of $\log \left|\zeta\left(\sigma+i\left(t+\Delta\right)\right)\right|$'s conditional distribution when $\zeta\left(\sigma+it\right)=0$ with $\sigma>1/2$:

\begin{thm}
Suppose $t$ is uniformly distributed in $[T,2T]$ with $T\rightarrow \infty$ and let $p^{*}_{t+\Delta|t}(x)$ denote the conditional probability density function of $\log\left|\zeta\left(\sigma+i\left(t+\Delta\right)\right)\right|$ given $\zeta\left(\sigma+it\right)$. Then
\begin{align}
E\left\{\log\left|\zeta\left(\sigma+i\left(t+\Delta\right)\right)\right||\zeta\left(\sigma+it\right)= 0\right\} = \begin{cases} 
       -\infty \hspace{.1cm}; & \textnormal{Re}\mathcal{P}\left(2\sigma+i\Delta\right)>0 \\
      +\infty \hspace{.1cm}; & \textnormal{Re}\mathcal{P}\left(2\sigma+i\Delta\right)<0\\
      O(1) \hspace{.1cm}; & \textnormal{Re}\mathcal{P}\left(2\sigma+i\Delta\right)=0
   \end{cases} \label{main result mean}
\end{align}
and, if $\sigma>1/2$ and $\zeta(\sigma+it)=0$, then
\begin{equation}
p^{*}_{t+\Delta|t}(x)= 0 \label{cor1 result}
\end{equation}
for all $|x|<\infty$ and all $\Delta \in \mathbb{R}$ such that $0<\left|\textnormal{Re}\mathcal{P}\left(2\sigma+i\Delta\right)/\mathcal{P}\left(2\sigma\right)\right|<1$.
\label{main theorem}
\end{thm}
\begin{proof}
We first apply (\ref{prime zeta ac}) and (\ref{over with}) to write
\begin{align}
\textnormal{Re}\mathcal{P}\left(\sigma+i\tau\right)=\log\left|\zeta\left(\sigma+i\tau\right)\right|+\sum_{n=2}^{\infty}\frac{\mu(n)}{n}\log \left|\zeta\left(n\sigma+ni\tau\right)\right| \nonumber \\ =\log\left|\zeta\left(\sigma+i\tau\right)\right|+O(1). \label{last step}
\end{align}
We next apply (\ref{gen cond mean}) with (\ref{cov gen}) to give the following result for the conditional expectation of $\textnormal{Re}\mathcal{P}\left(\sigma+i\left(t+\Delta\right)\right)$ given $\textnormal{Re}\mathcal{P}\left(\sigma+it\right)$:
\begin{equation}
E\left\{\textnormal{Re}\mathcal{P}\left(\sigma+i\left(t+\Delta\right)\right)|\textnormal{Re}\mathcal{P}\left(\sigma+it\right)\right\}=\frac{\textnormal{Re}\mathcal{P}\left(2\sigma+i\Delta\right)}{\mathcal{P}\left(2\sigma\right)}\textnormal{Re}\mathcal{P}\left(\sigma+it\right). \label{cond mean pz}
\end{equation}
We then note from (\ref{last step}) that when $\zeta(\sigma+it)=0$ with $\sigma>1/2$
\begin{equation}
\textnormal{Re}\mathcal{P}\left(\sigma+it\right)=\lim_{z \rightarrow 0^{+}}\log (z)+O(1), \label{pz divergence}
\end{equation}
which, with (\ref{cond mean pz}) and (\ref{always conv}), shows that
\begin{equation}
E\left\{\textnormal{Re}\mathcal{P}\left(\sigma+i\left(t+\Delta\right)\right)|\zeta\left(\sigma+it\right)=0\right\}\\ = \begin{cases} 
       -\infty \hspace{.1cm}; & \textnormal{Re}\mathcal{P}\left(2\sigma+i\Delta\right)>0 \\
      +\infty \hspace{.1cm}; & \textnormal{Re}\mathcal{P}\left(2\sigma+i\Delta\right)<0\\
      O(1)\hspace{.1cm}; & \textnormal{Re}\mathcal{P}\left(2\sigma+i\Delta\right)=0
   \end{cases} 
\label{almost main result mean}
\end{equation}
We additionally note from (\ref{last step}) that, for an arbitrary event $Y$, 
\begin{equation}
E\left\{\textnormal{Re}\mathcal{P}\left(\sigma+i\tau\right)|Y\right\}=E\left\{\log\left|\zeta\left(\sigma+i\tau\right)\right||Y\right\}+O(1). \label{arb}
\end{equation}
From the results (\ref{almost main result mean}) and (\ref{arb}) we conclude (\ref{main result mean}).

We then let $p_{t+\Delta|t}(x)$ denote the conditional probability density function of $\textnormal{Re}\mathcal{P}\left(\sigma+i\left(t+\Delta\right)\right)$ given $\textnormal{Re}\mathcal{P}\left(\sigma+it\right)$ and note that it has the elliptical form (\ref{cond dens gen}) with $\mu_{t+\Delta|t}$ given by (\ref{cond mean pz}) and
\begin{equation}
h(\Delta)=\frac{\mathcal{P}\left(2\sigma\right)}{2}\left(1-\left(\frac{\textnormal{Re}\mathcal{P}\left(2\sigma+i\Delta\right)}{\mathcal{P}\left(2\sigma\right)}\right)^{2}\right). \label{h pz}
\end{equation}
Due to the requirement that $\int_{-\infty}^{\infty} p_{t+\Delta|t}(x)dx$ converges, $g(u)=o(u^{-1/2})$. Hence, from (\ref{cond dens gen}) and (\ref{cond mean pz}), for nonzero (\ref{h pz}) we have
\begin{equation}
p_{t+\Delta|t}(x)=o\left(\left|x- \frac{\textnormal{Re}\mathcal{P}\left(2\sigma+i\Delta\right)}{\mathcal{P}\left(2\sigma\right)}\textnormal{Re}\mathcal{P}\left(\sigma+it\right)  \right|^{-1}\right). \label{cor1 done}
\end{equation}
(\ref{pz divergence})-(\ref{almost main result mean}) and (\ref{h pz})-(\ref{cor1 done}) show that $p_{t+\Delta|t}(x)=0$ if $\zeta\left(\sigma+it\right)=0$ for all $|x|<\infty$ and $\Delta$ such that $0 <\left|\textnormal{Re}\mathcal{P}\left(2\sigma+i\Delta\right)/\mathcal{P}\left(2\sigma\right)\right|<1$. Then from (\ref{last step}) the result (\ref{cor1 result}) follows.
\end{proof}

The proof of Theorem \ref{main theorem} shows that, for $\sigma>1/2$ and most values of $\Delta$, the conditional moments of $\log\left|\zeta\left(\sigma+i\left(t+\Delta\right)\right)\right|$ are infinite and its conditional probability density function is identically zero when $\zeta(\sigma+it)=0$. A probability density function cannot be zero everywhere on the real line. We thus conclude from this result that, for for most values of $\Delta$, the conditional distribution of $\log\left|\zeta\left(\sigma+i\left(t+\Delta\right)\right)\right|$ does not exist when $\zeta(\sigma+it)=0$.

Distances $\Delta$ where $\textnormal{Re}\mathcal{P}\left(2\sigma+i\Delta\right)=0$, i.e., where $\log\left|\zeta\left(\sigma+i\left(t+\Delta\right)\right)\right|$ and $\textnormal{Re}\mathcal{P}\left(\sigma+it\right)$ are \textit{uncorrelated} (see Appendix \ref{appb}) are a possible exception to the outcome described above. At these distances $\log\left|\zeta\left(\sigma+i\left(t+\Delta\right)\right)\right|$'s conditional expectation can exist, but the status of its higher order moments and density is unclear and depends on the specific, non-normal, unconditional distribution of $\textnormal{Re}\mathcal{P}\left(\sigma+it\right)$.

We recall from (\ref{always conv}) that (\ref{sq conv})-(\ref{density asy}) is satisfied by $\textnormal{Re}\mathcal{P}\left(\sigma+it\right)$ with $\sigma>1/2$, which makes it likely that $\textnormal{Re}\mathcal{P}\left(\sigma+it\right)$ and $\textnormal{Re}\mathcal{P}\left(\sigma+i\left(t+\Delta\right)\right)$ have a generalized hyperbolic distribution. If this is true, then we can easily prove that $\log\left|\zeta\left(\sigma+i\left(t+\Delta\right)\right)\right|$'s conditional density $p^{*}_{t+\Delta|t}(x)$ is also identically zero at distances $\Delta$ where $\textnormal{Re}\mathcal{P}\left(2\sigma+i\Delta\right)=0$. 

\begin{cor}
If $\textnormal{Re}\mathcal{P}\left(\sigma+it\right)$ and $\textnormal{Re}\mathcal{P}\left(\sigma+i\left(t+\Delta\right)\right)$ have a generalized hyperbolic distribution with $\sigma>1/2$ and $t$ uniformly distributed in $[T,2T]$ with $T\rightarrow \infty$, then, if $\textnormal{Re}\mathcal{P}\left(2\sigma+i\Delta\right)=0$ and $\zeta(\sigma+it)=0$, 
\begin{equation}
p^{*}_{t+\Delta|t}(x)= 0 \label{gen hyp density vanish}
\end{equation}
for all $|x|<\infty$.  \label{cor2}
\end{cor}
\begin{proof}
If generalized hyperbolic, then $\textnormal{Re}\mathcal{P}\left(\sigma+i\left(t+\Delta\right)\right)$'s conditional density has the form (\ref{gen hyp cond dens}) and, by (\ref{cond mean pz}), $\mu_{t+\Delta|t}=0$ for $\textnormal{Re}\mathcal{P}\left(2\sigma+i\Delta\right)=0$. Additionally from (\ref{gen hyp params}) it is clear that as $\left|\textnormal{Re}\mathcal{P}\left(\sigma+it\right)\right|\rightarrow \infty$ ,
\begin{equation}
\delta_{t+\Delta|t}\sim\frac{\left|\textnormal{Re}\mathcal{P}\left(\sigma+it\right)\right|}{R_{\mathcal{P}}(0)}. \label{del asy}
\end{equation}
Applying $\mu_{t+\Delta|t}=0$ and (\ref{del asy}) along with (\ref{m bessel identity}) in (\ref{gen hyp cond dens}) and simplifying shows that, for all $|x|<\infty$
\begin{equation}
p_{t+\Delta|t}(x)=O\left(\left|\textnormal{Re}\mathcal{P}\left(\sigma+it\right)\right|^{-1/2}\right),\label{cor2done}
\end{equation}
which, with (\ref{pz divergence}) and (\ref{last step}), proves (\ref{gen hyp density vanish}).
\end{proof}

We may also easily show that, under a generalized hyperbolic distribution assumption for $\textnormal{Re}\mathcal{P}\left(\sigma+it\right)$ and $\textnormal{Re}\mathcal{P}\left(\sigma+i\left(t+\Delta\right)\right)$, the conditional variance of $\log\left|\zeta\left(\sigma+i\left(t+\Delta\right)\right)\right|$ is infinite if $\zeta(\sigma+it)=0$ at distances $\Delta$ where $\textnormal{Re}\mathcal{P}\left(2\sigma+i\Delta\right)=0$ and indeed at any distances where the autocorrelation is not $\pm 1$.

\begin{prop}
If $\textnormal{Re}\mathcal{P}\left(\sigma+it\right)$ and $\textnormal{Re}\mathcal{P}\left(\sigma+i\left(t+\Delta\right)\right)$ have a generalized hyperbolic distribution with $\sigma>1/2$ and $t$ uniformly distributed in $[T,2T]$ with $T \rightarrow \infty$, then
\begin{equation}
\textnormal{var}\left\{\log\left|\zeta\left(\sigma+i\left(t+\Delta\right)\right)\right||\zeta(\sigma+it)=0\right\} = \infty  \label{main result var}
\end{equation}
for all $\Delta \in \mathbb{R}$ such that $\left|\textnormal{Re}\mathcal{P}\left(2\sigma+i\Delta\right)/\mathcal{P}(2\sigma)\right|<1$.  \label{var thm}
\end{prop}
\begin{proof}
We apply (\ref{gen hyp var asy}) with (\ref{cov gen}) and (\ref{always conv}) to note that
\begin{align}
\textnormal{var}\left\{\textnormal{Re}\mathcal{P}\left(\sigma+i\left(t+\Delta\right)\right)|\textnormal{Re}\mathcal{P}\left(\sigma+it\right)\right\} \nonumber \\ \sim \left|\textnormal{Re}\mathcal{P}\left(\sigma+it\right)\right|\frac{1}{\alpha}\sqrt{\frac{\mathcal{P}\left(2\sigma\right)}{2}}\left(1-\left(\frac{\textnormal{Re}\mathcal{P}\left(2\sigma+i\Delta\right)}{\mathcal{P}\left(2\sigma\right)}\right)^{2}\right)
 \label{var gen hyp asy}
\end{align}
as $\left|\textnormal{Re}\mathcal{P}\left(\sigma+it\right)\right| \rightarrow \infty$. The result (\ref{var gen hyp asy}) along with (\ref{always conv}) and (\ref{pz divergence}) shows that 
\begin{equation}
\textnormal{var}\left\{\textnormal{Re}\mathcal{P}\left(\sigma+i\left(t+\Delta\right)\right)|\zeta(\sigma+it)=0\right\} =\infty \label{pzeta var diverge}
\end{equation}
for all $\Delta \in \mathbb{R}$ such that $\left|\textnormal{Re}\mathcal{P}\left(2\sigma+i\Delta\right)/\mathcal{P}(2\sigma)\right|<1$. We next denote (\ref{last step})'s $O(1)$ term as $B\left(\sigma,\tau\right)$ and note that, for an arbitrary event $Y$,
\begin{align}
\textnormal{var}\left\{\textnormal{Re}\mathcal{P}\left(\sigma+i\tau\right)|Y\right\}=\textnormal{var}\left\{\log\left|\zeta\left(\sigma+i\tau\right)\right||Y\right\}+\textnormal{var}\left\{B\left(\sigma,\tau\right)|Y\right\} \nonumber \\ +2\textnormal{cov}\left\{\log\left|\zeta\left(\sigma+i\tau\right)\right|,B\left(\sigma,\tau\right)|Y\right\}. \label{big var 1}
\end{align}
We then apply the Cauchy-Schwartz inequality to (\ref{big var 1})'s last term to write
\begin{align}
\textnormal{var}\left\{\textnormal{Re}\mathcal{P}\left(\sigma+i\tau\right)|Y\right\}=\textnormal{var}\left\{\log\left|\zeta\left(\sigma+i\tau\right)\right||Y\right\}+\textnormal{var}\left\{B\left(\sigma,\tau\right)|Y\right\} \nonumber \\ +2\alpha\sqrt{\textnormal{var}\left\{\log\left|\zeta\left(\sigma+i\tau\right)\right||Y\right\} \textnormal{var}\left\{B\left(\sigma,\tau\right)|Y\right\}}, \label{big var 2}
\end{align}
where $|\alpha|\leq 1$. Since $B\left(\sigma,\tau\right)=O(1)$ it is then clear that 
\begin{equation}
\textnormal{var}\left\{\textnormal{Re}\mathcal{P}\left(\sigma+i\tau\right)|Y\right\}= \infty \implies \textnormal{var}\left\{\log\left|\zeta\left(\sigma+i\tau \right)\right||Y\right\}=\infty. \label{log zeta var diverge}
\end{equation}
The results (\ref{pzeta var diverge}) and (\ref{log zeta var diverge}) complete the proof of (\ref{main result var}).
\end{proof}

\subsection{Extension to L-functions}
L-functions $L\left(s,\chi\right)$ are an important generalization of $\zeta(s)$ that are useful for studying the distribution of primes in arithmetic progressions and related topics. Like $\zeta(s)$, these functions have a Dirichlet series and Euler product representation
\begin{equation}
L\left(s,\chi\right)=\sum_{n=1}^{\infty}\frac{\chi(n)}{n^{s}}=\prod_{p}\left(1-\frac{\chi(p)}{p^{s}}\right)^{-1}, \label{l function}
\end{equation}
where $\chi$ is a given Dirichlet character modulo $M$, which has the definition
\begin{align}
\chi(nm)=\chi(n)\chi(m), \nonumber \\
\chi(n)\begin{cases} =0; & \textnormal{gcd}(n,M)>1 \\ \neq 0; & \textnormal{gcd}(n,M)=1 \end{cases} \nonumber \\
\chi(n+M)=\chi(n) \label{dirichlet character}
\end{align}
If $\chi$ is a principal character, $\chi_{0}$, meaning it has the simple definition  
\begin{equation}
\chi_{0}(n)=\begin{cases} 0; & \textnormal{gcd}(n,M)>1 \\ 1; & \textnormal{gcd}(n,M)=1 \end{cases} \label{principal ch}
\end{equation}
then, similarly to $\zeta(s)$, $L(s,\chi_{0})$ may be extended by analytic continuation to all $s$ in $\mathbb{C}$ except $s=1$ where there is a simple pole. If $\chi$ is not principal, then $L(s,\chi)$ is an entire function. Also like $\zeta(s)$, it is hypothesized (the so-called generalized Riemann hypothesis) that all the the nontrivial zeros of any $L(s,\chi)$ reside on the line $\textnormal{Re}(s)=1/2$. 

We may generalize (\ref{prime zeta sum})-(\ref{prime zeta ac}) for the consideration of L-functions $L\left(s,\chi\right)$. We generalize (\ref{prime zeta sum}) by defining the prime L-function with the series 
\begin{equation}
\mathcal{P}_{\chi}\left(s\right)=\sum_{p}\frac{\chi(p)}{p^{s}} \label{primedirseries}
\end{equation}
where $\chi(p)=e^{-i\theta_{p}}$ for primes $p$ such that $p \nmid M$ and $\chi(p)=0$ for primes $p$ such that $p | M$. We next note that one may take the logarithm of $L(s,\chi)$'s Euler product, Taylor expand, and apply M{\"o}bius inversion to give a generalized version of (\ref{prime zeta ac}) relating L-functions $L(s, \chi)$ and their corresponding series $\mathcal{P}_{\chi}(s)$ for $\textnormal{Re}(s)>0$:
\begin{equation}
\mathcal{P}_{\chi}(s)=\sum_{n=1}^{\infty}\frac{\mu(n)}{n}\log L\left(ns,\chi^{n}\right) \label{prime l ac}
\end{equation}
By the same reasoning as that described in Section \ref{prelims}, the summands of (\ref{primedirseries})'s $\textnormal{Re}\mathcal{P}_{\chi}(\sigma+it)$ are statistically independent under $t$ uniformly distributed in $[a,b]$ with $b-a \rightarrow \infty$. It is also clear from the same reasoning that $\textnormal{Re}\mathcal{P}_{\chi}(\sigma+it)$ and $\textnormal{Re}\mathcal{P}_{\chi}\left(\sigma+i\left(t+\Delta\right)\right)$ have an elliptical joint distribution. Additionally, for $\sigma>1/2$,  $\textnormal{Re}\mathcal{P}_{\chi}(\sigma+it)$'s probability density function satisfies (\ref{density asy}), which is evidence that $\textnormal{Re}\mathcal{P}_{\chi}(\sigma+it)$ and $\textnormal{Re}\mathcal{P}_{\chi}\left(\sigma+i\left(t+\Delta\right)\right)$ with $\sigma>1/2$ have a generalized hyperbolic distribution.

Essentially equivalent reasoning to that in the proof of Lemma \ref{cov gen lemma} shows that, for $t$ uniformly distributed in $[T,2T]$ with $T\rightarrow \infty$
\begin{align}
E\left\{\textnormal{Re}\mathcal{P}_{\chi^{k}}\left(k\left(\sigma+i(t+\Delta)\right)\right)\textnormal{Re}\mathcal{P}_{\chi^{j}}\left(j\left(\sigma+it\right)\right)\right\} \nonumber \\ \rightarrow \begin{cases} \frac{1}{2}\textnormal{Re}\mathcal{P}_{\chi_{0}}\left(k(2\sigma+i\Delta)\right); &  k=j \\ 0; & k \neq j\end{cases} 
\label{pre covar general l functions}
\end{align}
Note the dependence on the principal character $\chi_{0}$. The result (\ref{pre covar general l functions}) gives $\textnormal{Re}\mathcal{P}_{\chi}(\sigma+it)$'s autocovariance function: 
\begin{equation}
R_{\mathcal{P}_{\chi}}\left(\sigma,\Delta\right)=E\left\{\textnormal{Re}\mathcal{P}_{\chi}\left(\sigma+i(t+\Delta)\right)\textnormal{Re}\mathcal{P}_{\chi}\left(\sigma+it\right)\right\}\rightarrow \frac{1}{2}\textnormal{Re}\mathcal{P}_{\chi_{0}}\left(2\sigma+i\Delta\right) \label{cov gen l}
\end{equation}
We can then apply (\ref{pre covar general l functions}) similarly to Lemma \ref{cov gen lemma} in the proof of Lemma \ref{logzeta cov lemma} to give $\log L\left(\sigma+it,\chi\right)$'s autocovariance function:
\begin{align}
R_{\log L_{\chi}}\left(\sigma,\Delta\right)=E\left\{\log L\left(\sigma+i\left(t+\Delta\right),\chi\right)\log L\left(\sigma+it,\chi\right)\right\} \nonumber \\
\rightarrow \frac{1}{2}\sum_{k=1}^{\infty}\frac{\textnormal{Re}\mathcal{P}_{\chi_{0}}\left(k\left(2\sigma+i\Delta\right)\right)}{k^{2}}
\label{log l cov}
\end{align}

We next make the important note that Selberg's central limit theorem (\ref{selberg clt}) also applies for L-functions on the critical line \cite{hsuwong}. $\log \left|L\left(1/2+it,\chi\right)\right|$ and $\log \left|L\left(1/2+i\left(t+i\Delta\right),\chi\right)\right|$ are therefore normally distributed under $t$ uniformly distributed in $[T,2T]$ with $T$ large and $\Delta \in [T-t,2T-t]$. We hence can apply the conditional mean formula (\ref{gen cond mean}) with (\ref{log l cov}) in the same way as we use (\ref{gen cond mean}) and (\ref{logzeta cov}) in the proof of Theorem \ref{thm intro}, noting from (\ref{primedirseries}), (\ref{prime l ac}), and an analogous result to (\ref{over with}) that
\begin{align}
\sum_{k=1}^{\infty}\frac{\mathcal{P}_{\chi_{0}}(k)}{k^{2}}=\mathcal{P}_{\chi_{0}}(1)+O(1)=\log L\left(1,\chi_{0}\right)+\sum_{n=2}^{\infty}\frac{\mu(n)}{n}\log L\left(n,\chi_{0}\right)+O(1)\nonumber \\ =\log L\left(1,\chi_{0}\right)+O(1)=\log\left(\lim_{z\rightarrow 1^{+}}\frac{1}{z-1}\right)+O(1)=\lim_{z\rightarrow 0^{+}}\log\left(\frac{1}{z}\right)+O(1), \label{l divergence}
\end{align}
since all principal L-functions have a simple pole at $s=1$. The divergent quantity (\ref{l divergence}) cancels with $\log\left|L\left(1/2+it,\chi\right)\right|$ when $L\left(1/2+it,\chi\right)=0$, which produces a finite result for $\log\left|L\left(1/2+i\left(t+\Delta\right),\chi\right)\right|$'s conditional mean. Lastly we can apply Selberg's central limit theorem (\ref{selberg clt}) for L-functions along with (\ref{log l cov}) in (\ref{variance factor}) and (\ref{gen cond var}) to show that the conditional variance is $\sim\frac{1}{2}\log\log T$. These results for the conditional mean and variance show that
\begin{align}
\left. \log \left|L\left(\frac{1}{2}+i\left(t+\Delta\right),\chi\right)\right| \middle \vert L\left(\frac{1}{2}+it,\chi\right)=0 \right.\nonumber \\ \overset{d} \to \mathcal{N}\left(-\sum_{k=1}^{\infty}\frac{\textnormal{Re}\mathcal{P}_{\chi_{0}}\left(k\left(1+i\Delta\right)\right)}{k^{2}}, \frac{1}{2}\log\log T\right) \label{critical line l}
\end{align}
for $t$ uniformly distributed in $[T,2T]$ with $\Delta \in [T-t,2T-t]$ and $T\rightarrow \infty$. This gives an analog of Theorem \ref{thm intro} for general L-functions.

Off the critical line we apply the same reasoning as in Theorem \ref{main theorem}'s proof, using (\ref{gen cond mean}) with (\ref{cov gen l}) and noting from (\ref{prime l ac}) that, if $L\left(\sigma+it,\chi\right)=0$, then
$$
\textnormal{Re}\mathcal{P}_{\chi}(\sigma+it)=\log\left|L\left(\sigma+it,\chi\right)\right|+O(1)=\lim_{z\rightarrow 0^{+}}\log(z)+O(1).
$$
This shows that, for $\sigma>1/2$, $t$ uniformly distributed in $[T,2T]$ with $T\rightarrow \infty$, and for all $\Delta \in \mathbb{R}$ such that $0<\left|\textnormal{Re}\mathcal{P}_{\chi_{0}}\left(2\sigma+i\Delta\right)/\mathcal{P}_{\chi_{0}}(2\sigma)\right|<1$:
\begin{equation}
E\left\{\log\left|L\left(\sigma+i\left(t+\Delta\right), \chi\right)\right||L\left(\sigma+it, \chi\right)=0\right\} = \pm \infty \label{main l result mean}
\end{equation}
and $\log\left|L\left(\sigma+i\left(t+\Delta\right), \chi\right)\right|$'s conditional probability density function given the value of $L\left(\sigma+it, \chi\right)$, denoted $p^{*}_{\chi,t+\Delta|t}(x)$, satisfies
\begin{equation}
p^{*}_{\chi,t+\Delta|t}(x)=0 \label{main l result density}
\end{equation}
if $L\left(\sigma+it, \chi\right)=0$ for all $|x|<\infty$. We conclude that, for $\sigma>1/2$ and most values of $\Delta$, the conditional distribution of $\log\left|L\left(\sigma+i\left(t+\Delta\right),\chi\right)\right|$ does not exist when $L\left(\sigma+it, \chi\right)=0$. This gives an analog of Theorem \ref{main theorem} for L-functions off the critical line. 

We may then apply similar reasoning to that in the proofs of Corollary \ref{cor2} and Proposition \ref{var thm} to show that, under a generalized hyperbolic distribution assumption for $\textnormal{Re}\mathcal{P}_{\chi}\left(\sigma+it\right)$ and $\textnormal{Re}\mathcal{P}_{\chi}\left(\sigma+i\left(t+\Delta\right)\right)$, the conditional density $p^{*}_{\chi, t+\Delta|t}(x)$ is also identically zero if $L\left(\sigma+it,\chi\right)=0$ at distances $\Delta$ where $\textnormal{Re}\mathcal{P}_{\chi_{0}}\left(2\sigma+i\Delta\right)=0$ and additionally 
\begin{equation}
\textnormal{var}\left\{\log\left|L\left(\sigma+i\left(t+\Delta\right), \chi\right)\right||L\left(\sigma+it,\chi\right)=0\right\}= \infty
\end{equation}
for all $\Delta \in \mathbb{R}$ such that $\left|\textnormal{Re}\mathcal{P}_{\chi_{0}}\left(2\sigma+i\Delta\right)/\mathcal{P}_{\chi_{0}}(2\sigma)\right|<1$. These results show that, similarly to the zeta function, the conditional distribution of an L-function's magnitude likely does not exist for any distances above or below nontrivial zeros off the critical line. 

\section{Conclusion}

The pole at $\zeta(1)$ is critically important to why $\log\left|\zeta\left(1/2+i\left(t+\Delta\right)\right)\right|$'s conditional distribution exists when $\zeta\left(1/2+it\right)=0$. Firstly, it provides an important cancellation in the conditional expectation formula that produces a convergent result around zeta zeros. Secondly, the pole-induced divergence of $\log\left|\zeta\left(1/2+it\right)\right|$'s variance, i.e., the fact that $\textnormal{var}\left\{\log\left|\zeta\left(1/2+it\right)\right|\right\} \sim \frac{1}{2}\log\log T$ for $t\in [T,2T]$ produces the central limit theorem and normal distribution on the critical line. The normal distribution implies that $\log\left|\zeta\left(1/2+i\left(t+\Delta\right)\right)\right|$'s conditional variance and higher order conditional moments do not depend on the value of $\zeta\left(1/2+it\right)$.

In contrast, with $\sigma>1/2$, there is no corresponding pole at $\zeta(2\sigma)$, so there is no stabilizing cancellation in $\log\left|\zeta\left(\sigma+i\left(t+\Delta\right)\right)\right|$'s conditional expectation formula and no central limit theorem providing a normal distribution. As a result the conditional moments and distribution of $\log\left|\zeta\left(\sigma+i\left(t+\Delta\right)\right)\right|$ do not exist if $\zeta\left(\sigma+it\right)=0$ except possibly for specific distances $\Delta$ where the correlation between $\textnormal{Re}\mathcal{P}\left(\sigma+it\right)$ and  $\textnormal{Re}\mathcal{P}\left(\sigma+i\left(t+\Delta\right)\right)$ vanishes. Further reseach is needed on these distances, however, we have shown that under relatively mild assumptions, the conditional distribution does not exist for these distances as well. All of this reasoning generalizes to L-functions.

We have thus shown that, although zeta and L-function magnitudes have well-understood vertical statistical structure around zeros on the critical line, they cannot have vertical statistical structure at most distances around nontrivial zeros off the critical line. This provides a novel, probabilistic explanation for why the Riemann hypothesis is likely to be true. Our proofs are relatively simple, involving the statistical properties of the log-Euler product and related prime sums, the conditional distribution structure of elliptical random variables, and the elementary pole structure of zeta and L-functions. Further research is needed to understand if these results could serve as a starting point for a proof of the Riemann hypothesis.

\appendix
\section{On statistical independence}
\label{app independence}

The characteristic function of a random variable $X$ is defined
\begin{equation}
\varphi_{X}(\lambda)=E\left\{e^{i\lambda X}\right\}. \label{chf}
\end{equation} 
We suppose $t$ is uniformly distributed in $[a,b]$ with $b-a \rightarrow \infty$ and consider the following sum over some set of primes $p$: 
\begin{equation}
X(t)=\sum_{p}a_{p}e^{-i\left(\alpha_{p}t\log p+\phi_{p}\right)}, \label{primedirseries m}
\end{equation}
where $a_{p}, \phi_{p} \in \mathbb{R}$ and $\alpha_{p} \in \mathbb{Q}$. We substitute (\ref{primedirseries m})'s real part into (\ref{chf}) to write
\begin{equation}
\varphi_{X(t)}(\lambda)=E\left\{ \prod_{p}\exp\left( i \lambda a_{p}\cos\left(\alpha_{p}t \log p+\phi_{p}\right)\right) \right\}. \label{app1step00}
\end{equation}
We then expand (\ref{app1step00}) using the Bessel function identity 
\begin{equation}
e^{ix \cos y}=\sum_{n=-\infty}^{\infty} i^{n}J_{n}\left(x\right)e^{i n y}
\label{besselidcos}
\end{equation}
where $J_{n}(.)$ is the $n$th-order Bessel function of the first kind. This gives
\begin{align}
\begin{split}
\varphi_{X(t)}\left(\lambda \right)=E\left\{\sum_{n_{1},n_{2},...}\left(i^{n_{1}+n_{2}+...}J_{n_{1}}\left(\lambda a_{p_{1}}\right)J_{n_{2}}\left(\lambda a_{p_{2}}\right)... \right. \right. \\ \left. \left. \times e^{i\left(n_{1}\phi_{p_{1}}+n_{2}\phi_{p_{2}}+...\right)}e^{i t\left(n_{1}\alpha_{p_{1}}\log p_{1}+n_{2}\alpha_{p_{2}}\log p_{2}+...\right)}\right)\right\} \end{split}
\label{big}
\end{align}
The exponential terms on (\ref{big})'s far right-hand side are unit circle rotations with $t$. Therefore taking the expected value will cause all terms to vanish except those for which 
\begin{equation}
n_{1}\alpha_{p_{1}}\log p_{1}+n_{2}\alpha_{p_{2}}\log p_{2}+n_{3}\alpha_{p_{3}}\log p_{3}+...=0. \label{zero}
\end{equation}
However, by unique-prime-factorization, the $\log p$'s are linearly independent over the rational numbers. Therefore the only solution to (\ref{zero}) is given by 
\begin{equation}
n_{1}=n_{2}=n_{3}=...=0. \label{soln}
\end{equation}
This simplifies (\ref{big}) to give 
\begin{equation}
\varphi_{X(t)}\left(\lambda \right)=\prod_{p} J_{0}\left(\lambda a_{p}\right). \label{result}
\end{equation}
We next note from (\ref{chf}) and (\ref{besselidsin}) that the characteristic function for a single summand in (\ref{primedirseries m})'s real part, $a_{p}\cos(\alpha_{p}t \log p+\phi_{p})$, is given by 
\begin{equation}
\varphi_{p}\left(\lambda\right)=J_{0}\left(\lambda a_{p}\right). \label{littleresult}
\end{equation}
Therefore, by (\ref{result}) and (\ref{littleresult}), 
\begin{equation}
\varphi_{X(t)}\left(\lambda \right)=\prod_{p} \varphi_{p}\left(\lambda\right). \label{independence}
\end{equation}
This shows that (\ref{primedirseries m})'s summands are independent. Essentially equivalent reasoning using the identity
\begin{equation}
e^{ix \sin \phi}=\sum_{n=-\infty}^{\infty}J_{n}\left(x\right)e^{i n\phi} \label{besselidsin}
\end{equation}
gives equivalent results for (\ref{primedirseries m})'s imaginary part. This completes the proof sketch. For further details and complete proof, apply results in \cite{laurincikas} (pg. 36, Cor. 5.2 pg. 38, Eqn. 6.4 pg. 41, Cor. 6.7 pg. 43, Cor. 6.8 pg. 44).

\section{Uncorrelated shifts}
\label{appb}
We consider $\sigma>1$ and apply the Euler product to write
\begin{align}
\log\left|\zeta\left(\sigma+i\left(t+\Delta\right)\right)\right|\textnormal{Re}\mathcal{P}\left(\sigma+it\right)=\sum_{k=1}^{\infty}\frac{\textnormal{Re}\mathcal{P}\left(k\left(\sigma+i\left(t+\Delta\right)\right)\right)\textnormal{Re}\mathcal{P}\left(\sigma+it\right)}{k} \label{covar trick 1}
\end{align}
We then note that the sum of the absolute values of (\ref{covar trick 1})'s summands has the upper bound 
\begin{equation}
\mathcal{P}(\sigma)\sum_{k=1}^{\infty}\frac{\mathcal{P}(k\sigma)}{k}<\infty, \label{covar trick 2 upper bound}
\end{equation}
which shows that (\ref{covar trick 1}) satisfies (\ref{ft cond}). We may therefore apply the expectation to (\ref{covar trick 1}) term-by-term. We do so under $a=T$, $b=2T$, and $T\rightarrow \infty$ and apply Lemma \ref{cov gen lemma} to show that the expected value of each term in (\ref{covar trick 1}) vanishes except the term corresponding to $k=1$. By (\ref{cov gen}), the expected value of the $k=1$ term gives
\begin{equation}
E\left\{\log\left|\zeta\left(\sigma+i\left(t+\Delta\right)\right)\right|\textnormal{Re}\mathcal{P}\left(\sigma+it\right)\right\}\rightarrow \frac{1}{2}\textnormal{Re}\mathcal{P}\left(2\sigma+i\Delta\right), \label{covar trick 4}
\end{equation}
which, as was noted after the proof of Lemma \ref{cov gen lemma}, has a domain extending to $\sigma>0$. Therefore, if $\textnormal{Re}\mathcal{P}\left(2\sigma+i\Delta\right)=0$, then $\log\left|\zeta\left(\sigma+i\left(t+\Delta\right)\right)\right|$ and $\textnormal{Re}\mathcal{P}\left(\sigma+it\right)$ are uncorrelated.


\begin{thebibliography}{1}

\bibitem{branch}Arguin, L.P., Belius, D. \& Harper, A.J. (2017). Maxima of a randomized Riemann zeta function, and branching random walks. \textit{The Annals of Applied Probability 27}(1), 178-215.

\bibitem{maxshort}Arguin, L.P., Belius, D., Bourgade, P., Radziwill, M. \& Soundararajan, K. (2019). Maximum of the Riemann zeta function on a short interval of the critical line. \textit{Communications on Pure and Applied Mathematics 72}(3), 500-535.

\bibitem{fhk 1}Arguin, L.P., Bourgade, P., \& Radziwill, M. (2020). The Fyodorov-Hiary-Keating Conjecture. I. arXiv:2007.00988.

\bibitem{largedev}Arguin, L.P., Ouimet, F. (2019). Large deviations and continuity estimates for the derivative of a random model of $\log|\zeta|$ on the critical line. \textit{Journal of Mathematical Analysis and Applications 472}(1), 687-695. 

\bibitem{shortgen}Arguin, L.P., Ouimet, F. \& Radziwill, M. (2021). Moments of the Riemann zeta function on short intervals of the critical line. \textit{The Annals of Probability 49}(6), 3106-3141.  


\bibitem{berry 77}Berry, M.V. \& Tabor, M. (1977). Level clustering in the regular spectrum. \textit{Proceedings of the Royal Society of London A 356}, 375-394. 

\bibitem{bessel}Bessel Functions, Digital Library of Mathematical Functions, National Institutes of Standards and Technology; \url{https://dlmf.nist.gov/10}.

\bibitem{blaesild81}Blaesild, P. \& Jensen  J.L. (1981). Multivariate Distributions of Hyperbolic Type. In: Taillie C., Patil G.P., Baldessari B.A. (eds) Statistical Distributions in Scientific Work. NATO Advanced Study Institutes Series (Series C: Mathematical and Physical Sciences), vol 79. Springer, Dordrecht.

\bibitem{bk951}Bogomolny, E. \& Keating, J. (1995). Random matrix theory and the Riemann zeros. I. Three- and four-point correlations. \textit{Nonlinearity 8}(6), 1115-1131. 

\bibitem{bk952}Bogomolny, E. \& Keating, J. (1996). Random matrix theory and the Riemann zeros. II. n-point correlations. \textit{Nonlinearity 9}, 911-935. 1115-1131. 

\bibitem{bk96}Bogomolny, E. \& Keating, J. (1996). Gutzwiller's trace formula and spectral statistics: beyond the diagonal approximation. \textit{Physical Review Letters 77}(8), 1472-1475. 

\bibitem{bohigas et al}Bohigas, O., Giannoni, M.J., Schmidt, C. (1984). Characterization of chaotic quantum spectra and universality of level fluctuation laws. \textit{Physical Review Letters 52}, 1-4. 


\bibitem{bourgade2010}Bourgade, P. (2010). Mesoscopic fluctuations of the zeta zeros. \textit{Probability Theory and Related Fields 148}, 479-500.

\bibitem{cfz}Conrey, J.B., Farmer, D.W., \& Zirnbauer, M.R. (2008). Autocorrelation of ratios of L-functions. \textit{Communications in Number Theory and Physics 2}(3), 593-636. 

\bibitem{conreysnaith}Conrey, J.B. \& Snaith, N.C. (2008). Correlations of eigenvalues and Riemann zeros. \textit{Communications in Number Theory and Physics 2}(3), 477-536.  

\bibitem{coram diaconis}Coram, M. \& Diaconis, P. (2003). New tests of the correspondence between unitary eigenvalues and the zeros of Riemann's zeta function. \textit{Journal of Physics A: Mathematical and General 36}(12), 2883-2906. 

\bibitem{dyson}Dyson, F.J. (1962). Statistical theory of the energy levels of complex systems. \textit{Journal of Mathematical Physics 3}, 140-175.

\bibitem{kotz}Fang, K.T., Kotz, S., \& Ng, K.W.  (1990). Symmetric Multivariate and Related Distributions. Springer-Science+Business Media. 

\bibitem{fordzah2} Ford, K. \& Zaharescu, A. (2004). On the distribution of imaginary parts of zeros of the Riemann zeta function. \textit{Journal f{\"u}r die Reine und Angewandte Mathematik 579}, 145-158. 

\bibitem{fordzah}Ford, K. \& Zaharescu, A. (2015). Unnormalized differences between zeros of L-functions. \textit{Compositio Mathematica 151}, 230-252.   

\bibitem{froberg}Fr{\"o}berg, C-E. (1968). On the Prime Zeta Function. \textit{Nordisk Tidskr. Informationsbehandling (BIT) 8}(3), 187-202. 

\bibitem{fhk}Fyodorov, Y.V., Hiary, G.A., \& Keating, J.P. (2012). Freezing transition, characteristic polynomials of random matrices, and the Riemann zeta function. \textit{Physical Review Letters 108}, 170601. 

\bibitem{fk}Fyodorov, Y.V. \& Keating, J.P. (2014). Freezing transitions and extreme values: random matrix theory, and disordered landscapes. \textit{Philosophical Transactions of the Royal Society A 372}, 20120503.

\bibitem{glaisher}Glaisher, J.W.L. (1891). On the Sums of Inverse Powers of the Prime Numbers. \textit{Quarterly Journal of Mathematics 25}, 347-362.

\bibitem{goneklandau} Gonek, S.M. (1993). An Explicit Formula of Landau and its Applications to the Theory of the Zeta-Function, A Tribute to Emil Grosswald: Number Theory and Related Analysis. \textit{Contemporary Mathematics 143}, 395-413. American Mathematical Society: Providence, RI.

\bibitem{gonekhugheskeating} Gonek, S.M., Hughes, C.P., \& Keating, J.P. (2007). A hybrid Euler-Hadamard product for the Riemann zeta function. \textit{Duke Mathematical Journal 136}(3), 507-549. 

\bibitem{gonek12}Gonek, S.M. (2012). Finite Euler products and the Riemann hypothesis. \textit{Transactions of the American Mathematical Society 364}(4), 2157-2191. 

\bibitem{gonek}Gonek, S.M. (2015). A note on finite Euler product approximations of the Riemann zeta-function. \textit{Proceedings of the American Mathematical Society 143}(8), 3295-3302. 

\bibitem{harper}Harper, A. (2013). A note on the maximum of the Riemann zeta function, and log-correlated random variables. arXiv:1304.0677.

\bibitem{hejhal}Hejhal, D. (1994). On the triple correlation between zeros of the zeta function. \textit{International Mathematics Research Notices}, 293-302.

\bibitem{hsuwong}Hsu, P. \& Wong, P. (2020). On Selberg's Central Limit Theorem for Dirichlet L-Functions. \textit{Journal de Th{\'e}orie des Nombres de Bordeaux 32}, 685-710.

\bibitem{katzsarnak}Katz, N.M. \& Sarnak, P. (1999). Zeros of zeta functions and symmetry. \textit{Bulletin of the American Mathematical Society 36}(1), 1-26.

\bibitem{keating snaith}Keating, J.P. \& Snaith, N.C. (2000). Random matrix theory and $\zeta(1/2+it)$. \textit{Communications in Mathematical Physics 214}, 57-89. 


\bibitem{landau}Landau, E. (1911). {\"U}ber die Nullstelen der Zetafunction. \textit{Mathematische Annalen 71}, 548-564. 

\bibitem{laurincikas}Laurin{\v c}ikas, A. (1996). Limit Theorems for the Riemann-Zeta Function. Dordrecht: Kluwer Academic Publishers. 

\bibitem{matsumoto}Matsumoto, K. (2006). An introduction to the value-distribution theory of zeta functions. \textit{Siauliai Mathematical Seminar 1}(9), 61-83.

\bibitem{montgomery}Montgomery, H. (1973). The pair correlation of zeros of the zeta function. \textit{Proceedings of the Symposium on Pure Mathematics 24}, 181-193. Providence, R.I.: American Mathematical Society. 

\bibitem{odlyzko}Odlyzko, A.M. (1987). On the distribution of spacings between zeros of the zeta function. \textit{Mathematics of Computation 48}, 273-308. 

\bibitem{pm11}Perez Marco, R. (2011). Statistics on the Riemann zeros. arXiv:1112.0346.

\bibitem{radsound}Radziwill, M. \& Soundararajan, K. (2017). Selberg's central limit theorem for $\log \left|\zeta(1/2+it)\right|$. \textit{L'enseignement Mathematique 63}(1/2), 1-19.

\bibitem{rodgers}Rodgers, B. (2013). Macroscopic pair correlation of the Riemann zeroes for smooth test functions. \textit{Quarterly Journal of Mathematics 64}(4), 1197-1219.

\bibitem{rudnicksarnak94}Rudnick, Z. \& Sarnak, P. (1994). The n-level correlations of zeros of the zeta function. \textit{C.R. Acad. Sci. Paris Ser. I. Math 319}, 1027-1032. 

\bibitem{rudnicksarnak96}Rudnick, Z. \& Sarnak, P. (1996). Zeros of principal L-functions and random matrix theory. \textit{Duke Mathematical Journal 81}, 269-322.

\bibitem{selberg}Selberg, A. (1946). Contributions to the theory of the Riemann zeta-function. \textit{Arch. Math. Naturvid. 48}(5), 89-155. 


\bibitem{snaith}Snaith, N.C. (2010). Riemann zeros and random matrix theory. \textit{Milan Journal of Mathematics 78}, 135-152.


\end{thebibliography}
\end{document}